\documentclass[11pt]{article}

\usepackage{amsmath,amsthm,verbatim,amssymb,amsfonts,amscd, graphicx}
\usepackage{graphics}

\usepackage[utf8]{inputenc}    
\usepackage[T1]{fontenc} 

\theoremstyle{plain}
\newtheorem{theorem}{Theorem}[section]
\newtheorem{lemma}{Lemma}[section]
\newtheorem{corollary}{Corollary}[section]
\newtheorem{proposition}{Proposition}[section]

\newtheorem*{claim*}{Claim}
\newtheorem*{lemma*}{Lemma}
\newtheorem*{theorem*}{Theorem}

\theoremstyle{definition}

\theoremstyle{remark}
\newtheorem{remark}{Remark}

\usepackage[
backend=bibtex,
style=authoryear-icomp,
sortlocale=de_DE,
natbib=true,
url=false, 
doi=true,
eprint=false,
dashed=false
]{biblatex}
\bibliography{2.20}

\begin{document}
	
	\title{Two applications of the Schwarz lemma}
	
	\author{Bingyuan Liu}

	\date{\today}

	\maketitle
	
	\begin{abstract}
		The Schwarz lemmas are well-known characterizations for holomorphic maps and we exhibit two examples of their applications. For a sequence family of biholomorphisms $f_j$, it is useful to determine the location of $f_j(q)$ for a fixed point $q$ in source manifolds (see Proposition \ref{2.5}). With it, we extend the Fornaess-Stout's theorem of \citep{FS77} in monotone unions of balls to ellipsoids in Section \ref{sec2}. In Section \ref{sec3}, we discuss the curvature bounds of complete K\"{a}hler metric on $\rtimes$ domains defined in \citep{Liu004} with an idea of \citep{Ya76}.
	\end{abstract}

	\section{Introduction}
	In \citep{FS77}, it was shown that if a m-dimensional complex (Kobayashi) hyperbolic manifold $M$ admits a monotone union of m-dimensional balls via $f_j$,  then the manifold $M$ is biholomorphic onto the unit ball $\mathbb{B}^m$.  Here, by $M$ admitting a monotone union of balls via $f_j$, we mean that 
	\begin{enumerate}
		\item there exists a sequence of open subsets $M_j\subset M$ so that $M_j\Subset M_{j+1}$,
		\item each $M_j$ is biholomorphic, by $f_j$, to the m-dimensional unit ball $\mathbb{B}^m$, and
		\item $M=\cup M_j$.
	\end{enumerate}
	In Section \ref{sec2}, we will follow this fashion and exhibit a theorem about monotone unions of ellipsoids $E_n:=\lbrace (z,w)\in\mathbb{C}^2: |z|^2+|w|^{2n}<1\rbrace$. Please note this topic is different with the work about automorphism groups and for the latter, readers are referred to \citep{BP91}, \citep{BP98}, \citep{GK91}, \citep{GK93}, \citep{Ro79} and \citep{Wo77}.
	
	Kobayashi metrics have a lot of applications and interesting overlaps with Teichm\"{u}ller metrics. However, it is not very easy to be calculated. Especially, after different geometric flows were introduced, the Kobayashi metric is difficult to be manipulated by differential equations. So in this article, we will assume $M$ to be a complex manifold with the holomorphic sectional curvature bounded from above by a negative number. The readers should be warned that this condition is slightly stronger than (Kobayashi) hyperbolicity due to a theorem of Greene-Wu in \citep{GW79}. 
	
	Let $M$ be a hermitian complex manifold with a holomorphic sectional curvature bounded from above by a negative number and $\Omega$ be any complete K\"{a}hler manifold with the Ricci curvature bounded from below. Then all holomorphic maps $F$ from $\Omega$ into $M$ follow a rule, namely, the Schwarz lemma. For the detail and background, we will introduce them in the Section \ref{sec1}.  The reader should note provided a complete Bergman metric on $M$ and replacing ellipsoids with strongly pseudoconvex domains in the Theorem \ref{thm1}, one can show $M$ is biholomorphic to a ball easily by an argument like \citep{Kl78} but we will not discuss it here (see the Remark \ref{more}).  The results in Section \ref{sec2} are expected to hold for (Kobayashi) hyperbolic $M$, but, again, we will not discuss it here because it will be not about Schwarz lemma.
	
	The other application in this note is about the curvature bounds. After the Schwarz lemma of \citep{Ya78}, Yang was able to show, by Yau's Schwarz lemma, that there does not exit a complete K\"{a}hler metric on polydiscs admitting holomorphic bisectional curvature bounded above and below. This argument was used to show product manifolds in \citep{SZ08} and the argument was also polished by \citep{Se12}. For almost-Hermitian manifold, please refer to \citep{To07} and \citep{FTY14}. In Section \ref{sec3}, we show a type of domains which are not biholomorphic, in general, to product manifolds but have the similar properties as the one Yang discovered. 
	
	\section{Preliminary and fundamental facts}\label{sec1}
	
	Royden's Schwarz lemma generalized the classical Schwarz lemma and Ahlfors' Schwarz lemma. In this note, we will use this version of Schwarz lemma. 
	
	\begin{theorem}[the Schwarz lemma of \citep{Ro80}]
		Let $f: M\rightarrow N$ be a holomorphic mapping from a complete K\"{a}hler manifold $(M, g)$ with its Ricci curvature bounded from below by a negative constant $-k$ into a Hermitian manifold $(N, h)$ with its holomorphic sectional curvature bounded from above by a negative constatnt $-K$. If $\nu$ is the maximal rank of the map $f$, then \begin{equation*}
		f^*h\leq\frac{2k\nu}{K(\nu+1)}g.
		\end{equation*}
	\end{theorem}
	
	In Section \ref{sec2}, we show a theorem about unions of ellipsoids via $f_j$. Basically, we want to find a biholomorphism on the source manifolds by passing to the limit of $f_j$.  However, the most difficult part is that there is an interior pint $q$ and a series biholomorphism $f_j$ so that $f_j(q)$ approaches to the boundary. In this case, the limit of $f_j$ will be a map with degenerated Jacobian, rather than a biholomorphism. The essential part to resolve this difficulty is using Schwarz lemmas to find relation between $Jf_j(q)$ and the location of $f_j(q)$. With this, we can composite each $f_j$ with a injective holomorphic map $\phi_j$ so that the Jocobian of composition $\det J \phi_j\circ f_j$ has non vanishing limit. And then we can discuss the biholomorphisms $\phi_j\circ f_j$ with their limit.
	
	The other application in this note is exhibited in the Section \ref{sec3}. It is obtained from modifying the well-known Yang's argument. For convenience, we give some preliminaries here. The interested readers are referred to \citep{KL11}. The following theorem was used in Yang's argument.

	\begin{theorem}[the almost maximal principles of \citep{Ya78}]
		Let $M$ be a complete Riemannian manifold $M$ with the Ricci curvature bounded from below, then for any $C^2$ smooth function $f:M\rightarrow\mathbb{R}$ that is bounded from above, there exists a sequence $\lbrace p_k\rbrace$ such that 
		\begin{equation*}
		\lim\limits_{k\to\infty}|\nabla T(p_k)|=0, \quad\limsup\limits_{k\to\infty}\Delta T(p_k)\leq 0 \quad \text{and} \quad\lim\limits_{k\to\infty}T(p_k)=\sup\limits_{M}T.
		\end{equation*}
	\end{theorem}

	We introduce some terminology. Let $(M, J, h)$ be a K\"ahler mainifold $M$ of dimension m with a K\"{a}hler metric $h$ and a complex structure $J$. The curvature tensor $R$ on $(M, J, h)$ is given by
	\begin{equation*}
	R_{i\bar{j}k\bar{l}}=\dfrac{\partial^2h_{i\bar{j}}}{\partial z_k\partial\bar{z}_l}-\sum_{\alpha,\beta=1}^{m}h^{\alpha\bar{\beta}}\dfrac{\partial h_{i\bar{\beta}}}{\partial z_k}\dfrac{\partial h_{\alpha\bar{j}}}{\partial\bar{z}_l}
	\end{equation*}
	in local coordinates $(z_1,...,z_n)$. The holomorphic sectional curvature for $X\in T_pM$ at $p\in M$ is given by 
	\begin{equation*}
	B(X)=-\dfrac{\sum_{i,j,k,l=1}^{m}R_{i\bar{j}k\bar{l}}X_i\bar{X}_jX_k\bar{X}_l}{(\sum_{i,j=1}^{m}h_{i\bar{j}}X_i\bar{X}_j)^2},
	\end{equation*}
	where 
	\begin{equation*}
	X=\sum_{j=1}^{m}X_j\frac{\partial}{\partial z_j}+\sum_{j=1}^{m}\bar{X}_j\frac{\partial}{\partial \bar{z}_j}.
	\end{equation*}
	
	\section{Monotone unions of Ellipsoids}\label{sec2}
	In this section, we discuss monotone unions of ellipsoids $E_n:=\lbrace(z,w): |z|^2+|w|^{2n}<1\rbrace$ in $\mathbb{C}^2$. Take an arbitrary point $q\in M$, and let $j\rightarrow\infty$, then $f_j(q)$ has a limit point, possibly after passing to a subsequence, because the image of $f_j$ is $\Omega$ which is bounded. Then the location of limits of $f_j(q)$ has two different cases, an interior point of $\Omega$ or a boundary point at $\partial\Omega$.
	
	The following lemma settle the case that the limit of $f_j(q)$ is an interior point of $\Omega$. From now on, we will not distinguish the convergence and the subsequence. We hope it will not confuse the reader.
	
	\begin{lemma}\label{0}
		Let $M$ be a m-dimensional complex manifold with a holomorphic sectional curvature bounded from above by a negative number $-K$ and assume $M$ is a monotone union of $\Omega\subset\mathbb{C}^m$ via $f_j$ where $\Omega$ is a bounded domain in $\mathbb{C}^m$ with a complete K\"{a}hler manifold with the Ricci curvature bounded from below by a negative number $-k$. We also assume there exists an interior point $q\in M$ so that $f_j(q)\rightarrow p\in\Omega$ then $M$ is biholomorphic onto $\Omega$.
	\end{lemma}
	
	\begin{proof}
		We first show that $M$ is biholomorphic into $\Omega$. Since $\Omega$ is bounded, $f_j$ is a normal family of biholomorphisms. Let $f_j$ converge to a holomorphic map $F$. Considering the inverses $\lbrace f^{-1}_j\rbrace_{j=1}^\infty$, we want to show they are locally bounded at least in a small geodesic ball $B_p$ centered $p\in\Omega$ with radius $\epsilon>0$. Indeed, by the Schwarz lemma in \citep{Ro80}, 
		\begin{equation*}
		(f^{-1}_j)^*\mathrm{d}_M\leq C\mathrm{d}_\Omega
		\end{equation*}
		for each $j>0$, where $C=\dfrac{2m}{m+1}\dfrac{k}{K}$. Let $N>0$ be big so that $f_j(q)\in B_p$ for all $j>N$. We also take arbitrarily $z\in B_p$ and now we have
		\begin{equation}\label{1}
		\mathrm{d}_M(q, f_j^{-1}(z))\leq C\mathrm{d}_\Omega(f_j(q), z)<2C\epsilon,
		\end{equation}
		for $j>N$. This means $f_j^{-1}$ is locally bounded in $B_p$ and we denote $G$ by its limit. One can see now $F\circ G(z)=z$ in $B_p$ because both limits of $f_j$ and $f_j^{-1}$ are uniformly convergent on compact subsets of $B_p$. It implies $\det Jf_j(q)\not\rightarrow 0$ and hence the limit of  $(\det Jf_j)(z)$ is nowhere vanishing for arbitrary $z\in M$, where $J$ denotes the Jacobian. The reason is what follows. By the Cauchy estimates, the fact that $\lbrace f_j\rbrace_{j=1}^\infty$ implies that $\lbrace\det Jf_j\rbrace_{j=1}^\infty$ is also normal. But $\lbrace\det Jf_j\rbrace_{j=1}$ is nowhere zero for each $j>0$ because $f_j$ is biholomorphism and then by Hurwitz theorem, $\det JF$ is a zero function or nowhere zero. And the conclusion follows by the fact that $\det Jf_j(q)\not\rightarrow 0$. Now $\det Jf_j(z)\not\rightarrow 0$ for all $z\in M$ and hence $\det JF(z)$ is nonzero everywhere which also implies $F(M)$ is open by the open mapping theorem. 
		
		It is the time to show $F$ is also 1-1. Indeed, otherwise, there are two interior points $z_0, w_0\in M$ so that $\mathrm{d}_\Omega(f_j(z_0), f_j(w_0))\rightarrow 0$ and we consider the Schwarz lemma of Royden \citep{Ro80} for $f_j^{-1}$ again,
		\begin{equation}\label{2}
		\mathrm{d}_M(z_0,w_0)\leq C\mathrm{d}_\Omega(f_j(z_0), f_j(w_0)).
		\end{equation}
		Since $\det Jf_j(z)$ does not approach to zero for all $z\in M$, $f(z)$ does not approach to a boundary point. Specifically, $f_j(z_0), f_j(w_0)$ do not approach to the boundary of $\Omega$ where the K\"{a}heler metric blows up, and so $\mathrm{d}_\Omega(f_j(z_0), f_j(w_0))\rightarrow 0$ implies $z_0=w_0$ by the Equation (\ref{2}). So $F$ is 1-1 and $M$ is a taut manifold.
		
		We can show now $M$ is biholomorphic onto $\Omega$. For this, we just need to show $G$ is 1-1. We use a well-known argument here. Again by $F$ we denote the limit of $f^j$ and by $G$, we denote the limit of $f_j^{-1}$. Since $M$ is taut, both of them make sense now. Suppose $z',z''\in\Omega$ and $G(z')=G(z'')$,
		\begin{equation*}
		\begin{split}
		z'-z''=&f_j\circ f^{-1}_j(z')-f_j\circ f^{-1}_j(z'')\\=&f_j\circ f^{-1}_j(z')-F\circ f^{-1}_j(z')+F\circ f^{-1}_j(z')-F\circ f^{-1}_j(z'')+F\circ f^{-1}_j(z'')-f_j\circ f^{-1}_j(z'').
		\end{split}
		\end{equation*}
		Let $j\rightarrow\infty$. Since both $\lbrace f_j^{-1}\rbrace_{j=1}^\infty$ and $\lbrace f_j\rbrace_{j=1}^\infty$ are normal, we have $z'=z''$.
	\end{proof}
	
	By the similar argument, we can verify the following corollary. Instead of looking at only the exhaustive subsets of $M$ in the previous lemma, the following corollary consider both exhaustive subsets of $M$ and $\Omega$. It is also a key to the problem of monotone unions of ellipsoids.
	
	\begin{corollary}\label{0.5}
		Let $M$ be a m-dimensional complex manifold with holomorphic sectional curvature bounded from above by a negative number $-K$ and assume $M=M_j$ where $M_j\subset M_{j+1}$ and $f_j$ is biholomorphism from $M_j$ onto $\Omega_j\subset\Omega\subset\mathbb{C}^m$. Suppose $\Omega$ is a bounded domain in $\mathbb{C}^m$ and $\Omega_j$ is a complete K\"{a}hler manifold with the Ricci curvatures bounded from below by a same negative number $-k$ independent with $j$. We also assume there exists a  point $q\in M$ so that $\det Jf_j(q)\not\rightarrow 0$ then $F$ is 1-1 and hence $M$ is taut.
	\end{corollary}
	For the sake of completeness, we also include a short outline of proof.
	\begin{proof}[Outline of  proof]
		Since $\Omega$ is bounded, $\Omega_j\subset\Omega$ is bounded too for each $j>0$. Hence $\lbrace f_j\rbrace_{j=1}^\infty$ is still normal. By the $\det J f_j(q)\not\rightarrow 0$, we can see $\det Jf_j(z)\not\rightarrow 0$ everywhere for $z\in M$, where $\lbrace\det Jf_j(z)\rbrace_{j=1}^\infty$ is normal because of the Cauchy estimates. This means, for any $z\in M$, $f_j(z)$ does not approach to $\partial\Omega$. So by the Schwarz lemma of \citep{Ro80}, we find the limit $F$ of $f_j$ is 1-1. Moreover, this means $M$ is taut.
	\end{proof}
	
	The Lemma \ref{0} and Corollary \ref{0.5} tell us that if there exists one point $q$ such that $f_j(q)\rightarrow p\in\Omega$, then for any point $z\in M$, we have $f_j(z)$ approaching to an interior point of $\Omega$. Furthermore, the limit of $f_j$ forms a biholomorphism. However, this is not the only case. Indeed, sometimes $f_j(q)$ can approach to a boundary point of $\Omega$ and this brings trouble for getting the biholomorphism. For example if $\Omega$ is of finite type, then the image of $F=\lim\limits_{j\to\infty}f_j$ will be a constant function which of course cannot be the biholomorphic map we look for. The reason why the limit is a constant map is the $\det Jf_j(q)\rightarrow 0$. So we need to composite each $f_j$ with a biholomorphic map $\phi_j$ so that the result map $\det J\phi_j\circ f_j$ has a nonzero limit. To find the appropriate $\phi_j$ we need to estimate the speed of decay for $\det Jf_j(q)$. It appears the speed of decay can be arbitrary, but indeed, the decay is constrained by the location of $f_j$ due to an application of the Schwarz lemma as follows.

	\begin{proposition}\label{2.5}
		Let $M$ be a m-dimensional complex manifold with the holomorphic sectional curvature bounded from above by a negative number $-K$ and assume $M$ is a monotone union of $\Omega\subset\mathbb{C}^m$ via $f_j$ where $\Omega$ is a bounded domain in $\mathbb{C}^m$ with a complete Bergman metric with the Ricci curvature bounded from below by a negative number $-k$. We also assume there exists a point $q\in M$ so that $f_j(q)\rightarrow p\in\partial\Omega$ where $p$ is strongly pseudoconvex. Then $\dfrac{|Jf_j(q)|}{\delta(f_j(q))^{\frac{m+1}{2}}}\gtrsim\eta$ for some positive $\eta$, where $\delta$ is the distance function of $\Omega$.
	\end{proposition}
	\begin{proof}
		For $f^{-1}_j$, by the Schwarz lemma of \citep{Ro80}, we have $(f^{-1}_j)^*g_M\leq Cg_\Omega$ for some constant $C$ where $g_M$ is the metric on $M$ and by $g_\Omega$, we denote the Bergman metric of $\Omega$. In local coordinates, we have for any tangent vector $X_{o}\in\mathrm{T}_{o}\Omega$ at $o\in\Omega$
		\begin{equation*}
		((f^{-1}_j)_*X_{o})'G_M(f^{-1}_j(o))(f^{-1}_j)_*X_{o}\leq C X'_{o}G_{\Omega}(o)X_{o},
		\end{equation*}
		where we denote the conjugate transpose by $'$ and matrices of $g_M$ and $g_\Omega$ by $G_M$ and $G_\Omega$ respectively. For each $j>0$, we let $o=f_j(q)$ and we have
		\begin{equation*}
		((f^{-1}_j)_*X_{f_j(q)})'G_M(f^{-1}_j(f_j(q)))(f^{-1}_j)_*X_{f_j(q)}\leq C X'_{f_j(q)}G_{\Omega}(f_j(q))X_{f_j(q)},
		\end{equation*}
		Without loss of the generality, we pick up the coordinates on $M$ at $q$ so that $G_m$ is identity matrix. Hence $ (Jf^{-1}_j(f_j(q)))'Jf^{-1}_j(f_j(q))\leq C G_{\Omega}(f_j(q))$ and by Minkowski determinant theorem, we also have 
		\begin{equation}\label{2.1}
		|\det Jf^{-1}_j(f_j(q))|^2\leq C |\det G_{\Omega}(f_j(q))|.
		\end{equation} 
		But $G_\Omega$ is a metric around a strongly pseudoconvex point $p$, so by \citep{Fe74}, it is equivalent to the $\partial\bar{\partial}(\log\delta)$ up to nonzero constant. Moreover, by computation the second order Taylor expansion of $\delta$ at $p$, we also have 
		\begin{equation*}
		|\det G_\Omega(o)|\leq\dfrac{c_0}{\delta(o)^{m+1}}
		\end{equation*} 
		for some $c_0>0$, when $o$ close to $p$. Again, put $o=f_j(q)$, we have
		\begin{equation}\label{2.2}
		|\det G_\Omega(f_j(q))|\leq\dfrac{c_0}{\delta(f_j(q))^{m+1}}
		\end{equation} 
		for sufficient big $j>0$. Since $\det Jf^{-1}_j(f_j(q))\det Jf_j(q)=\mathrm{id}$, we have, by the Equation (\ref{2.1}) and (\ref{2.2}), that $\dfrac{|\det Jf_j(q)|}{\delta(f_j(q))^\frac{m+1}{2}}>\dfrac{1}{\sqrt{c_0C}}$ for sufficient $j>0$. We let $\eta=\dfrac{1}{\sqrt{c_0C}}$, and thus get the desired result.
	\end{proof}
	
	One of the main techniques in this paper was motivated by a simple observation in one variable. Specifically, a small disc can approaches to a bigger disc by certain M\"{o}bius transforms of the bigger disc.
	
	\begin{lemma}[Two discs lemma]
		Suppose there is a faimily of M\"{o}bius transfrom on the unit disc $\psi_j(z)=\dfrac{z+\alpha_j}{1+\bar{\alpha}_jz}$ where $\alpha_j\in\mathbb{R}$ and $\alpha_j\rightarrow 1$. We also define a small disc $\mathcal{D}_s$ center at $b\in\mathbb{R}$ with radius $1-b$ where $b>0$ close to $1$. Then $\phi^{-1}_j(\mathbb{D}_s)\rightarrow\mathbb{D}$ in sense of convergence in increasing subsets.
	\end{lemma}
	\begin{proof}
		By the condition, we want to get $\lbrace z\in\mathbb{C}^2: |f_j(z)-b|<1-b\rbrace$ for each $j$.
		\begin{equation*}
		\begin{split}
		&|\dfrac{z+\alpha_j}{1+\bar{\alpha}_jz}-b|<1-b\\\Leftrightarrow&|z+\alpha_j-b-\bar{\alpha}_jbz|^2<(1-b)^2|1+\bar{\alpha }_jz|^2\\\Leftrightarrow&|z+\dfrac{(\alpha_j-b)(1-\alpha_jb)-(1-b)^2\alpha_j}{|1-\bar{\alpha}_jb|^2-(1-b)^2|\alpha_j|^2}|^2<\dfrac{|1-b|^2-|\alpha_j-b|^2}{|1-\bar{\alpha}_jb|^2-(1-b)^2|\alpha_j|^2}\\&+\dfrac{|(\alpha_j-b)(1-\alpha_jb)-(1-b)^2\alpha_j|^2}{(|1-\bar{\alpha}_jb|^2-(1-b)^2|\alpha_j|^2)^2}.
		\end{split}
		\end{equation*}
		Thus for the j-th step, it is a disc centered at \begin{equation*}
		o_j=-\dfrac{(\alpha_j-b)(1-\alpha_jb)-(1-b)^2\alpha_j}{|1-\bar{\alpha}_jb|^2-(1-b)^2|\alpha_j|^2}
		\end{equation*} 
		with radius
		\begin{equation*}
		r_j=\sqrt{\dfrac{|1-b|^2-|\alpha_j-b|^2}{|1-\bar{\alpha}_jb|^2-(1-b)^2|\alpha_j|^2}+\dfrac{|(\alpha_j-b)(1-\alpha_jb)-(1-b)^2\alpha_j|^2}{(|1-\bar{\alpha}_jb|^2-(1-b)^2|\alpha_j|^2)^2}}.
		\end{equation*}
		Let us calculate the limit of $o_j$,
		\begin{equation*}
		\lim_{j\rightarrow\infty}\dfrac{(\alpha_j-b)(1-\alpha_jb)-(1-b)^2\alpha_j}{|1-\bar{\alpha}_jb|^2-(1-b)^2|\alpha_j|^2}=\lim_{x\rightarrow 1}\dfrac{(x-b)(1-xb)-(1-b)^2x}{|1-xb|^2-(1-b)^2|x|^2}=0,
		\end{equation*}
		by L'H\^{o}pital's rule. For the same reason, $r_k\rightarrow 1$ as $j\rightarrow\infty$.
	\end{proof}
	
	The imitation to balls are also available.

	\begin{lemma}[Two balls lemma]
		Suppose there are a faimily of automorphisms \begin{equation*}
		\psi_j(z,w)=(\dfrac{z+a_j}{1+\bar{a}_jz},\dfrac{\sqrt{1-|a_j|^2}}{1+\bar{a}_jz}w)
		\end{equation*} 
		on the unit ball $\mathbb{B}^m$ where $\alpha_j\in\mathbb{R}$ and $\alpha_j\rightarrow 1$. We also define a small ball $\mathcal{B}_s$ in the same dimension center at $(b,0)\in\mathbb{R}$ with radius $1-b$ where $b>0$ close to $1$. Then $\psi^{-1}_j(\mathbb{B}_s)\rightarrow\mathbb{B}^m$ in sense of convergence in increasing subsets.
	\end{lemma}
	
	\begin{proof}
		Since $\mathbb{B}_s=\lbrace(z',w'): |z'|^2+|w'|^2<1\rbrace$, we have
		\begin{equation*}
		|\dfrac{z+a_j}{1+\bar{a}_jz}-b|^2+|\dfrac{\sqrt{1-|a_j|^2}}{1+\bar{a}_jz}w|^2<(1-b)^2.
		\end{equation*}
		By calculation, we have 
		\begin{equation}\label{2.75}
		\begin{split}
		&|z+\dfrac{(a_j-b)(1-a_jb)-(1-b)^2a_j}{|1-\bar{a}_jb|^2-(1-b)^2|a_j|^2}|^2+|\dfrac{\sqrt{1-|a_j|^2}}{\sqrt{|1-\bar{a}_jb|^2-(1-b)^2|a_j|^2}}w|^2<\dfrac{(1-b)^2-|a_j-b|^2}{|1-\bar{a}_jb|^2-(1-b)^2|a_j|^2}\\&+\dfrac{|(a_j-b)(1-a_jb)-(1-b)^2a_j|^2}{(|1-\bar{a}_jb|^2-(1-b)^2|a_j|^2)^2}.
		\end{split}
		\end{equation}
		Again, by L'H\^{o}pital's rule, one can see the formula in the Equation (\ref{2.75}) approaches to 
		\begin{equation*}
		|z|^2+|w|^2<1
		\end{equation*}
	\end{proof}

	\begin{theorem}\label{thm1}
		Let $M$ be a 2-dimensional complex manifold with holomorphic sectional curvature bounded from above by a negative number $-K$ and assume $M$ is a monotone union of ellipsoids $E_n:=\lbrace (z,w): |z|^2+|w|^{2n}<1\rbrace$ for some $n\in\mathbb{Z}^+$ via $f_j$. Then $M$ is biholomorphic onto $E_n$ or the unit ball $\mathbb{B}^2$.
	\end{theorem}
	\begin{remark}\label{more}
		Indeed, when $f_j(q)$ approaches to a strongly pseudoconvex point, for some $q\in M$, one can have a more general result, by using the argument of \citep{Kl78} provided a complete Bergman metric on $M$. However, we will not discuss it because it will be not an application of the Schwarz lemma.
	\end{remark}
	\begin{proof}
		If there exists a point $q\in M$ so that  $f_j(q)\rightarrow p\in E_n$ where $p$ is an interior point of $E_n$ then by the Lemma \ref{0}, $M$ is biholomorphic to $E_n$. Now we analyze the cases that $f_j(q)\rightarrow p\in\partial E_n$.
		
		It is well known on $\partial E_n$, there are only two types of boundary points: weakly pseudoconvex points $(e^{i\theta},0)$ where $\theta\in[0,2\pi)$ and strongly pseudoconvex on all other boundary points. 
		
		Let $f_j(q)\rightarrow p$ where $p$ is weakly pseudoconvex. Without loss of generality, we assume $p=(1,0)$. For each $f_j$ and $f_j(q)=(a_j, b_j)$, we composite it with an automorphism $\psi_j$ of $E_n$ so that $(a_j, b_j)$ maps to $(0, b'_j)$ for some $b'_j\in\mathbb{R}$. Indeed, this is possible by letting 
		\begin{equation*}
		\psi_j(z,w)=(\dfrac{z-a_j}{1-\bar{a}_jz}, e^{-i\theta_j}\dfrac{\sqrt[{2n}]{1-|a_j|^2}}{\sqrt[n]{1-\bar{a}_jz}}w),
		\end{equation*}
		where $\theta_j$ is the argument of $b_j$ and thus $b'_j=\dfrac{\sqrt[{2n}]{1-|a_j|^2}}{\sqrt[n]{1-\bar{a}_ja_j}}|b_j|$. Since $b_j'$ is bounded, it must have limit $b'_0$. If $b'_0\neq 1$, $\psi_j\circ f_j(q)=\psi_j(a_j, b_j)=(0,b'_j)\rightarrow (0, b'_0)$, where $(0, b'_0)$ is an interior point of $E_n$. By the lemma \ref{0}, we finish constructing the biholomorphism by passing $\psi_j\circ f_j$ to the limit and in this case $M$ is biholomorphic to $E_n$. For the case of $b'_0=1$, we get $(0, b'_j)\rightarrow (0,1)$ which means it approaches to a strongly pseudoconvex point. We discuss it in the next paragraph.
		
		Before we proceed to the case that $p$ is strongly pseudoconvex, we simplify it a little bit. If $f_j(q)=(a_j,b_j)\rightarrow p=(a_0, b_0)$, where $b_0\neq 0$, then we composite each $f_j$ with $\psi_j$ which maps $(a_j, b_j)$ to $(0, b'_j)$ for some $b_j\in\mathbb{R}$ as the last paragraph. But $\psi_j\circ f_j(q)\rightarrow (0, b'_0)$ where $(0, b'_0)$ is not necessarily to be an interior point.  If it is interior point of $\Omega$, by the Lemma \ref{0}, we get again $M$ is biholomorphic to $E_n$, otherwise, $\psi_j\circ f_j(q)$ approaches to a boundary point $(0, b'_0)$ which is strongly pseudoconvex. Without loss of generality, we assume $b'_0=1$. From now on, we write $\tilde{f}_j$ instead of $\psi_j\circ f_j$.
		
		The ellipsoid, by translation, has a defining function
		\begin{equation*}
		|z|^2+|w-1|^{2n}<1\Leftrightarrow |w|^2+1-2\Re w< \sqrt[n]{1-|z|^2}\Leftrightarrow |w|^2+\dfrac{1}{n}|z|^2+o(|z|^2)<2\Re w. 
		\end{equation*}
		On the other side the ball $\mathcal{B}_l$ center at $(0,1)$ with radius $1$ has a defining function $|z|^2+|w|^2<2\Re w$ and the ball $\mathcal{B}_s$ center at $(0, n)$ with radius $n$ has a defining function $\dfrac{1}{n}|z|^2+\dfrac{1}{n}|w|^2<2\Re w$. So $\mathcal{B}_s\subset E_n\subset\mathcal{B}_l$ and they are tangent to each other at $(0,0)$. Without loss of generality, we assume $\mathcal{B}_l=\mathbb{B}^2$ by translation and zooming. We also know $f_j(q)=(0, b'_j)$ where $b'_j\rightarrow 1$ as $j\rightarrow\infty$. By the Lemma \ref{2.5}, we see $|J\tilde{f}_j(q)|\gtrsim\eta\delta(\tilde{f}_j(q))^{\frac{m+1}{2}}$, which implies in our case that $|J\tilde{f}_j(q)|\gtrsim\eta(1-\|\tilde{f}_j(q)\|)^{\frac{m+1}{2}}=\eta(1-|b'_j|)^{\frac{m+1}{2}}$. We define 
		\begin{equation*}
		\phi_j=(\dfrac{\sqrt{1-|b'_j|^2}}{1+|b'_j|w}z, \dfrac{w+|b'_j|}{1+|b'_j|w})
		\end{equation*}
		and 
		\begin{equation*}
		\phi^{-1}_j=(\dfrac{\sqrt{1-|b'_j|^2}}{1-|b'_j|w}z,\dfrac{w-|b'_j|}{1-|b'_j|w},).
		\end{equation*}
		Hence 
		\begin{equation*}
		\det J\phi^{-1}_j\circ \tilde{f}_j(q)=\det J\phi^{-1}(\tilde{f}_j(q))\det J\tilde{f}_j(q)\gtrsim\eta(1-|b'_j|^2)^{-3}(1-|b'_j|^2)^3.
		\end{equation*} 
		But the last term never vanishes. Thus the limit $F$ of $\phi^{-1}_j\circ \tilde{f}_j$ has nontrivial image. Moreover, the image of $F$ is $\mathbb{B}^2$ because by the two balls lemma $\phi^{-1}_j(\mathcal{B}_s)\subset\phi^{-1}(E_n)=\phi^{-1}_j(\tilde{f}_j(M_j))$ and $\phi^{-1}_j(\mathcal{B}_s)$ is growing to $\mathbb{B}^2$. 
		
		At the last, we check the injectivity of $F$. Firstly, the Bergman metric on $E_n$ is of invariance under $\phi^{-1}_j$. Thus the discussion of the Bergman metric on $\phi^{-1}_j(E_n)$ makes sense. Since $\det JF$ is nowhere vanishing, for any $z_0, w_0\in M$ so that $\lim\phi_j^{-1}\circ \tilde{f}_j(z_0)=\lim \phi_j^{-1}\circ \tilde{f}_j(w_0)$ we have a $N>0$ so that  for all big $j$, $\phi_j^{-1}\circ \tilde{f}_j(z_0), \phi_j^{-1}\circ \tilde{f}_j(w_0)\in\phi^{-1}_N(E_n)$ and
		\begin{equation*}
		\mathrm{d}_M(z_0, w_0)\leq C\mathrm{d}_{\phi^{-1}_N(E_n)}(\phi^{-1}_j(f_j(z_0)) \phi^{-1}_j(f_j(w_0))).
		\end{equation*}
		Hence $F$ is injective which completes the proof.
	\end{proof}
	
	\begin{remark}
		The proof above also gives a example of a unit ball which is a monotone union of ellipsoids by letting $M_j=\phi^{-1}_j(E_m)$.
	\end{remark}

	Without much effort, one can show the following corollary.

	\begin{corollary}
		Let $M$ be a m-dimensional complex manifold with holomorphic sectional curvature bounded from above by a negative number $-K$ and assume $M$ is a monotone union of balls with the same dimension, then $M$ is biholomorphic onto $\mathbb{B}^m$.
	\end{corollary}

	\section{An application to $\rtimes$ domains}\label{sec3}
	In \citep{Liu004}, the author defined a generalized bidisc $\mathbb{D}\rtimes e^{i\theta(z)}\mathbb{H}^+:=\lbrace (z,w): z\in\mathbb{D}, w\in e^{i\theta(z)}\mathbb{H}^+\rbrace$, where $\mathbb{D}$ is the unit disc, $\mathbb{H}^+$ is the upper half plane, $\theta$ is a continuous real function depending on $z$ and $e^{i\theta(z)}\mathbb{H}^+$ is the upper half plane rotated by the angle $\theta(z)$. It has a noncompact automorphism group and share some properties with the bidisc. Indeed, when $\theta(z)$ is a zero function, $\mathbb{D}\rtimes e^{i\theta(z)}\mathbb{H}^+$ is biholomorphic to a bidisc.
	
	In this section, we will use a well-known argument of \citep{Ya76} to exhibit there does not exist a complete K\"{a}hler metric with holomorphic bisectional curvature between two negative numbers. Indeed, we can extend it a little more with the Schwarz lemma of \citep{Ro80} as follows.
	
	\begin{theorem}
		Let $\theta(z)\in[0,k)$, where $k<\pi$. Then there does not exist two numbers $d>c>0$ and a complete K\"{a}hler metric on $\mathbb{D}\rtimes e^{i\theta(z)}\mathbb{H}^+$ such that the holomorphic sectional curvature $<-c$ and the Ricci curvature $>-d$. 
	\end{theorem}
	
	Indeed, although Yang's argument has certain requirement on both variables of $\lbrace (z,w): z,w\in\mathbb{D}\rbrace$, it is possible to relax the requirement for the second variable in our proof. Of course similar results for higher dimensions hold for the same reason. But our intention here is only to attract more attention to $\rtimes$ domains and will not exhaust all possibilities. We will also use a modified argument from \citep{Se12}.
	
	\begin{proof}
		We assume the conclusion is not true. Let us denote the Poincar\'{e} metric by $g$ and the complete K\"{a}hler metric on $\mathbb{D}\rtimes e^{i\theta(z)}\mathbb{H}^+$ by $h$. For each $z$, we define $i_z(w)=(z, ie^{i\theta(z)}\dfrac{1+w}{1-w})$ from $\mathbb{D}$ onto $e^{i\theta(z)}\mathbb{H}^+$. We get $i^*h\leq \dfrac{4}{c} g$ because the Ricci curvature of $\mathbb{D}$ is $-4$. Thus,
		\begin{equation}\label{eq1}
		\begin{split}
		&\begin{pmatrix}
		0& \dfrac{2ie^{i\theta(z)}}{(1-w)^2}
		\end{pmatrix}
		\begin{pmatrix}
		h_{11}(z,ie^{i\theta}\dfrac{1+w}{1-w})&h_{12}(z,ie^{i\theta}\dfrac{1+w}{1-w})\\h_{21}(z,ie^{i\theta}\dfrac{1+w}{1-w})&h_{22}(z,ie^{i\theta}\dfrac{1+w}{1-w})
		\end{pmatrix}
		\begin{pmatrix}
		0\\\dfrac{-2ie^{-i\theta(z)}}{(1-\bar{w})^2}
		\end{pmatrix}\\=&h_{22}(z,ie^{i\theta(z)}\dfrac{1+w}{1-w})\dfrac{4}{|1-w|^4}\leq\dfrac{4}{c(1-|w|^2)^2}
		\end{split}
		\end{equation}
		and we have $h_{22}(z,ie^{i\theta(z)}\dfrac{1+w}{1-w})\leq\dfrac{|1-w|^4}{c(1-|w|^2)^2}\leq\dfrac{16}{c(1-|w|^2)^2}$.
		Since $k<\pi$, we find $\epsilon>0$ such that $k+\epsilon<\pi$. And because of $0\leq\theta(z)<k$, for any $z\in\mathbb{D}$, $(z,e^{i(k+\frac{\epsilon}{2})})\in\mathbb{D}\rtimes e^{i\theta(z)}\mathbb{H}^+$. We also have, for all $z\in\mathbb{D}$,
		\begin{equation}\label{eq2}
		\dfrac{\epsilon}{2}<k+\dfrac{\epsilon}{2}-\theta(z)<k+\dfrac{\epsilon}{2}<k+\epsilon<\pi.
		\end{equation}
		We let $w=\dfrac{e^{i(k+\frac{\epsilon}{2})-\theta(z)}-i}{e^{i(k+\frac{\epsilon}{2}-\theta(z))}+i}$ and by the inequality (\ref{eq2}), we can see $|1-|w||>\eta>0$ for some positive number $\eta$ depending on $\epsilon$. Also by the inequality (\ref{eq1}) and $w$, we have
		\begin{equation*}
		h_{22}(z,e^{i\theta(z)}e^{i((k+\frac{\epsilon}{2})-\theta(z))})=h_{22}(z,e^{i(k+\frac{\epsilon}{2})})\leq\dfrac{16}{c\eta^2}.
		\end{equation*}
		Let $F(z):=h_{22}(z,e^{i(k+\frac{\epsilon}{2})})$. We see $F$ is a real bounded positive function on $\mathbb{D}$. Check its Laplacian with respect to Poincar\'{e} metric on $\mathbb{D}$, we have
		\begin{equation*}
		\begin{split}
		&\Delta_gF(z)=(1-|z|^2)^2\dfrac{\partial^2F}{\partial z\bar{\partial}z}(z)=(1-|z|^2)^2(R_{2\bar{2}1\bar{1}}(z,e^{i(k+\frac{\epsilon}{2})})+\sum_{\alpha,\beta=1}^{2}h^{\alpha\bar{\beta}}\dfrac{\partial h_{2\bar{\beta}}}{\partial z}\dfrac{\partial h_{\alpha\bar{2}}}{\partial\bar{z}})\\\geq&c(1-|z|^2)^2h_{2\bar{2}}(z,e^{i(k+\frac{\epsilon}{2})})h_{1\bar{1}}(z,e^{i(k+\frac{\epsilon}{2})})=cF(z)(1-|z|^2)^2h_{1\bar{1}}(z,e^{i(k+\frac{\epsilon}{2})}),
		\end{split}
		\end{equation*}
		because $\sum_{\alpha,\beta=1}^{2}h^{\alpha\bar{\beta}}\dfrac{\partial h_{2\bar{\beta}}}{\partial z}\dfrac{\partial h_{\alpha\bar{2}}}{\partial\bar{z}}$ is nonnegative. Let $\pi:\mathbb{D}\rtimes e^{i\theta(z)}\mathbb{H}^+\rightarrow\mathbb{D}, \pi(z,w)=z$. We also have $\pi^*g\leq\dfrac{d}{4}h$ which is $(1-|z|^2)^2h_{1\bar{1}}(z,w)\leq \dfrac{4}{d}$. Hence $\Delta_gF(z)\geq\dfrac{c}{d}F$. Calculating 
		\begin{equation*}
		\Delta_g\log F(z)=\dfrac{\Delta_gF(z)}{F(z)}-\dfrac{|\nabla_g F(z)|^2}{F(z)^2}\geq\dfrac{2c}{d}-\dfrac{|\nabla_g F(z)|^2}{F(z)^2}.
		\end{equation*}
		By the alomost maximum principle of \citep{Ya78}, a real function $T$ bounded from above on a complete Riemannian manifold $M$ with Ricci curvature bounded below admits a sequence $\lbrace p_k\rbrace_{k=0}^\infty\subset M$ such that
		\begin{equation*}
		\lim\limits_{k\to\infty}|\nabla T(p_k)|=0, \quad\limsup\limits_{k\to\infty}\Delta T(p_k)\leq 0 \quad \text{and} \quad\lim\limits_{k\to\infty}T(p_k)=\sup\limits_{M}T.
		\end{equation*}
		However by observing $\log F(z)$ which is a real function bounded from above on $\mathbb{D}$, it can not have such sequence $\lbrace p_k\rbrace_{k=0}^\infty\subset\mathbb{D}$. This contradiction completes the proof.
	\end{proof}
	
	\begin{remark}
		A natural question is if we can relax the restriction for $\theta(z)$ in the theorem above. 
	\end{remark}
	
	\bigskip
	\bigskip
	
	\noindent {\bf Acknowledgments}. I thank a lot my advisor Prof. Steven Krantz for always being patience to answer my any (even stupid) question. He also encourages me a lot for the career and the life as both of a supervisor and a raconteur. I also thank Prof. Quo-Shin Chi who taught me various geometries and spent much time with me (to help me understand Ricci flow, algebraic geometry and some other interesting topics). Last but not least, I appreciate Prof. Edward Wilson for his kindness and consistent support since I arrived in USA. I always have much profited from the discussion between us. 
	
\printbibliography

\end{document}